\documentclass[10pt]{article}
\usepackage[mathlines]{lineno}
\usepackage[left=2cm,right=2cm,top=2cm,bottom=2cm]{geometry}
\usepackage{mathtools,amssymb,latexsym,amsthm,amsmath,amsfonts}   
\usepackage{xcolor}
\usepackage{authblk}
\newtheorem{theorem}{Theorem}
\newtheorem{lemma}{Lemma}

\newtheorem{definition}{Definition}
\newcommand{\Li}{{\mbox{Lip}}}
\newcommand{\cS}{\mathcal{S}}
\newcommand{\cC}{\mathcal{C}}
\newcommand{\cD}{{\cal D}}
\newcommand{\cP}{\mathcal{P}}

\newcommand{\ux}{\underline{x}}
\newcommand{\uy}{\underline{y}}
\newcommand{\uz}{\underline{z}}
\newcommand{\puy}{\partial_{\underline{y}}}

\newcommand{\bC}{{\bf C}}
\newcommand{\so}{\scriptscriptstyle\mathcal{O}}
\newcommand{\uzz}{\underline{\zeta}}
\newcommand{\f}{\tilde{f}}
\newcommand{\g}{\tilde{g}}
\newcommand{\R}{\mathbb{R}}
\newcommand{\C}{\mathbb{C}}
\newcommand{\E}{{\bf E}}
\title{Decomposition of first order Lipschitz functions by Clifford algebra-valued harmonic functions}
\author[1]{Lianet De la Cruz Toranzo}
\author[2,3]{Ricardo Abreu Blaya}	
\author[1]{Swanhild Bernstein}

\affil[1]{Department of Mathematics and Informatics, TU Bergakademie Freiberg, Germany.}
\affil[2]{Facultad de Matem\'aticas, Universidad Aut\'onoma de Guerrero, M\'exico.}
\affil[3]{Investigador Invitado. Universidad UTE, Ecuador.}

\begin{document}
\maketitle
\begin{abstract}\noindent
In this paper we solve the problem on finding a sectionally Clifford algebra-valued harmonic function, zero at infinity and satisfying certain boundary value condition related to higher order Lipschitz functions. Our main tool are the Hardy projections related to a singular integral operator arising in bimonogenic function theory, which turns out to be an involution operator on the first order Lipschitz classes. Our result generalizes the classical Hardy decomposition of H\"older continuous functions on a simple closed curve in the complex plane.
\end{abstract}

\section{Introduction} 
A classical boundary value problem in Complex analysis is to find a sectionally holomorphic function $\Phi(z)$, zero at infinity and satisfying the given boundary condition $\Phi(t)^+-\Phi(t)^-=\varphi(t),$ $\,t\in\Gamma,$
where $\varphi$ satisfies the H\"older condition on $\Gamma$, i.e. $\varphi\in\bC^{0,\alpha}(\Gamma)=\{f: \vert f(t)-f(\tau)\vert\leq c\vert t-\tau\vert^\alpha;\,t,\tau\in \Gamma\},\,0<\alpha\leq1.$ Here $\Gamma$ is a closed smooth curve in the complex plane and $\Phi(t)^\pm$ are the limiting values as approaching the curve from the interior $\Omega_+$ and the exterior $\Omega_-$ domain respectively.
This problem is immediately solved by the Cauchy transform,
\begin{equation*}\label{Def_TC}
\cC\varphi(z):=\frac{1}{2\pi i}\int_{\Gamma}\frac{\varphi(\zeta)}{\zeta-z}d\zeta, z\in \C\setminus\Gamma, 
\end{equation*}
which has continuous limiting values as approaching $\Gamma$ from the interior or exterior domain \cite{Ga,Mu}. Those involve both the identity and the singular integral operator
\begin{equation*}\label{Def_S}
\cS \varphi(t)=p.v.\frac{1}{\pi i}\!\int_{\Gamma}\!\frac{\varphi(\zeta)}{\zeta-t}d\zeta=\lim\limits_{\epsilon\to 0} \frac{1}{\pi i}\!\int_{\Gamma\setminus\Gamma_\epsilon}\!\!\frac{\varphi(\zeta)}{\zeta-t}d\zeta,\,\, t\in \Gamma,\,\,\Gamma_\epsilon=B_\epsilon(t)\cap \Gamma
\end{equation*}
and they are expressed by the classical Plemelj-Sokhotski formula: \begin{equation}\label{Plemelj-Sokhotski}
\begin{cases} 
\displaystyle
\cC^+\varphi(t)=\lim\limits_{z\to t\atop z\in\Omega_+}\cC\varphi(z)
=\frac{1}{2}\varphi(t)+\frac{1}{2\pi i}\int_{\Gamma}\frac{\varphi(\zeta)}{\zeta-t}d\zeta
=\frac{1}{2}\big[I+\cS\big]\varphi,\\
\displaystyle
\cC^-\varphi(t)=\lim\limits_{z\to t\atop z\in\Omega_-}\cC\varphi(z)=
-\frac{1}{2}\varphi(t)+\frac{1}{2\pi i}\int_{\Gamma}\frac{\varphi(\zeta)}{\zeta-t}d\zeta
=\frac{1}{2}\big[-I+\cS\big]\varphi.
\end{cases}
\end{equation}

The limiting values of the Cauchy transform with density function satisfying the H\"older condition, also verify this \cite{Ga,Mu}. Therefore, the singular integral operator keeps invariant the H\"older class (Plemelj-Privalov theorem) and it yields the Hardy decomposition of H\"older functions: $\bC^{0,\alpha}(\Gamma)=\bC^{0,\alpha}(\Gamma)^+\oplus \bC^{0,\alpha}(\Gamma)^-,$ where the uniqueness may be seen as a consequence of the involution property for the singular operator, i.e. $\cS^2=I$.\\

In the present work, we aim to obtain a similar decomposition for first order Lipschitz functions in the framework of Clifford analysis instead, which is a multidimensional function theory where the so-called monogenic functions play a similar roll to that by the holomorphic ones in Complex analysis. Our main result is Theorem \ref{Th_inv} and it will be stated and proved in Section \ref{Sec_Main}.

\section{Preliminaries}
\subsection{Clifford algebras and Clifford analysis}
Denote by $\{e_i\}_{i=1}^m$ an orthonormal basis of $\R^m$ governed by the multiplication rules
\begin{eqnarray}
e_i^2=-1,\quad e_ie_{j}=-e_{j}e_i,\,\, \forall i,j=1,2,\dots, m;\,\,i<j.
\end{eqnarray} 
Then the real Clifford algebra $\R_{0,m}$ is generated by $\{e_i\}_{i=1}^m$ over the field of real numbers $\R$, so that an element $a\in\R_{0,m}$ may be written as $$a=\sum_{A} a_A e_A,$$ where $a_A\in\R$ and $A$ runs over all the possible ordered sets $A=\{1\le i_1<\cdots < i_k\le m\}$ or $A=\emptyset$ and $e_A:=e_{i_1}e_{i_2}\cdots e_{i_k},\,\,e_0=e_\emptyset=1.$ If we identify $\ux\in\R^m$ by $\ux=x_1e_1+\dots +x_me_m;\,x_i\in\R,i=\overline{1,m}$ then we easily see that $\R^m$ is embedded in $\R_{0,m}$.\\

A norm for $a\in \R_{0,m}$ may be defined by $\|a\|^{2}=\sum_A|a_A|^{2},$ so that the Clifford algebra can be seen as a $2^n$-Euclidean space with Euclidean metric. In particular, for $\ux\in\R^m$ we have $\|\ux\|=|\ux|$, where $|\cdot|$ denotes the usual Euclidean norm.\\

We will consider functions defined on subsets of $\R^m$ and taking values
in $\R_{0,m}$, namely $f=\sum_{A} f_A e_A$. We shall say that an $\R_{0,m}$-valued function belongs to certain classical class of functions, if each of its real components $f_A$ do so.

From now on, $\Omega$ stands for a Jordan domain, i.e. a bounded oriented connected open subset of $\R^{m}$ whose boundary is a compact topological surface. For simplicity, we assume $\partial\Omega$ to be sufficiently smooth, e.g. Liapunov surface. By $\Gamma$ we denote the boundary of $\Omega$, while $\Omega_+=\Omega$, $\Omega_-=\R^m\setminus{\Omega\cup\Gamma}$ if necessary.\\

An $\R_{0,m}$-valued function $f$ in $\bC(\Omega)$ is called left (resp. right) monogenic if $\cD_{\ux}\,f = 0$ (resp. $f\,\cD_{\ux} = 0$) in $\Omega$, where $\cD_{\ux}$ is the Dirac operator
\[
\cD_{\ux}={\partial_{x_1}}e_1+{\partial_{x_2}}e_2+\cdots +{\partial_{x_m}}e_m.
\]

It is easy to check that the so-called Clifford-Cauchy kernel $E_0 (\ux)=-\displaystyle\frac{1}{\sigma_{m}}\frac{\ux}{|\ux|^{m}}\,\,(\ux\neq0),$
where $\sigma_{m}$ stands for the surface area of the unit sphere in $\R^{m}$, is a two-sided monogenic function.\\

If $f\in\bC^1(\Omega_+\cup\Gamma)$ is monogenic in $\Omega_+$, then the following representation formula holds
\[
f(\ux)=\int_\Gamma E_0(\uy-\ux)n(\uy)f(\uy)d\uy,\quad \ux\in\Omega_+,
\]
where $n(\uy)$ stands for the outward pointing unit normal at $\uy\in\Gamma.$ \\

This formula gives rise to the cliffordian Cauchy type transform and its singular version 
\[
\cC_0=\int_\Gamma E_0(\uy-\ux)n(\uy)f(\uy)d\uy,\, (\ux\in\Omega_+)\quad \mbox{and}\quad \cS_0=2\int_\Gamma E_0(\uy-\ux)n(\uy)f(\uy)d\uy,\,(\ux\in\Gamma
).\]

When $f$ satisfies the H\"older condition then, as in \eqref{Plemelj-Sokhotski}, both operators are connected by the cliffordian Plemelj-Sokhotski formula \cite{Ift}, namely 
\begin{equation}\label{cliff_Plemelj-Sokhotski_C0}
[\cC_0f]^+=\frac{1}{2}[I+\cS_0]f\quad \mbox{and }\quad [\cC_0f]^-=\frac{1}{2}[-I+\cS_0]f.
\end{equation}
Note that $\cD_{\ux}$ factorizes the Laplace operator $\triangle_m$ in $\R^m$ in the sense that $
\cD_{\ux}^2=-\triangle_m.$ The fundamental solution of $\cD_{\ux}$ is thus given by $E_0(\ux)=\cD_{\ux} E_1(\ux),$ where
\[
E_1 (\ux)=\frac{1}{ (m-2)\sigma_m{|\ux|}^{m-2}}\,\,(\ux\neq 0,\,m>2)
\]
is the fundamental solution of the Laplacian in $\R^m$ ($m\ge 3$).\\

In addition, an $\R_{0,m}$-valued function $f$ in $\bC^k(\Omega)$ is called polymonogenic of order $k$ or simply $k$-monogenic (left) if it satisfies $\cD_{\ux}^k\,f = 0$ in $\Omega$. 

\subsection{Higher order Lipschitz functions and Whitney extension theorem}
Given an appropriate space of functions defined on a non-empty closed set $\E\subset\R^m$, how can these be extended to $\R^m$? The study of these problems and linear properties of some related extension operators in terms of Banach spaces, led to the higher order Lipschitz functions,  defined by E. Stein in \cite{St}, but going back to the works of H. Whitney \cite{Wh}, who introduced a notion of differentiation in a general set ($\bC^m$-continuity in terms of $f_{(k)}$). The most appropriate spaces are the ones given in terms of the modulus of continuity, in particular, the Lipschitz spaces $\Li(\alpha,\E)$ with exponent $0<\alpha\leq 1$ composed by those functions $f$ defined on $\E$ such that 
\begin{eqnarray}\label{Def_Holder}
|f(x)|\leq M\quad \mbox{and}\quad |f(x)-f(y)|\leq M|x-y|^\alpha\quad\mbox{for}\quad x,\,y\in\E,	
\end{eqnarray}
which are Banach spaces with the smallest $M$ in the above definition as norm. In fact, the linear extension operator ${\cal{E}}_0$ defined by
\begin{equation}\label{Def_op_E0}
{\cal{E}}_0(f)(x)=\begin{cases}
	f(x),&x\in\E\\ \sum\limits_{i}f(p_i)\varphi_i^*(x),&x\in\E^c
\end{cases}
\end{equation}
maps $\Li(\alpha,\E)$ continuously into $\Li(\alpha,\R^m)$ for $0<\alpha\leq 1$ and the norm is independent of the closed set $\E$. This definition deals with the ideas on decomposition of open sets into cubes and partition of the unity as well. It calls for more explanation, but it is beyond the scope of this paper. For more details we refer the reader to \cite[Ch. VI]{St}.\\

When $\alpha>1$, the space $\Li(\alpha,\E)$ is constituted by constants only. However, it can be generalized as follows.
\begin{definition}[Higher order Lipschitz functions]
Let $\E$ be a closed subset of $\R^m$, $k$ a nonnegative integer and $0<\alpha\le 1$. We shall say that a real valued function $f$, defined in $\E$, belongs to the higher order Lipschitz class $\Li(k+\alpha,\E)$ if there exist real valued functions  $f^{(j)}$, $0<|j|\le k$, defined on $\E$, with $f^{(0)}=f$, and so that if 
\begin{equation*}
f^{(j)}(x)=\sum_{|j+l|\le k}\frac{f^{(j+l)}(y)}{l !}(x-y)^{l}+R_j(x,y),\,\,x, y\in\E
\end{equation*}
then
\begin{equation*}
|f^{(j)}(x)|\le M,\,\,\,|R_j(x,y)|\le M|x-y|^{k+\alpha-|j|},\,\,x, y\in\E, |j|\le k,
\end{equation*}
where $M$ is a positive constant and $j= (j_1,\dots ,j_m)$ and $l= (l_1,\dots ,l_m)$ are $m$-dimensional multi-indices in ${\mathbb{N}}_0^m$.
\end{definition}
We recall the multi-indices conventions
\[
\ux^j := x_1^{j_1}\cdots x_m^{j_m},\,j!:=j_1!\cdots j_m!,\,|j|:=j_1+\cdots + j_m,\,\partial^{j}:=\frac{\partial^{|j|}}{\partial^{j_1}_{ x_1}\dots \partial^{j_m}_{ x_m}}.
\]

First, some remarks concerning this definition are in order. We note that the function $f^{(0)}=f$ does not necessarily determine the elements $f^{(j)}$ for an arbitrary $\E$. Therefore, an element of $\Li(k+\alpha, \E)$ has to be interpreted as a collection of functions $\{f^{(j)}:\E\mapsto\R,\,|j|\le k\}$. However, if $\E=\R^m$ and $f\in\Li(k+\alpha,\R^m)$, then according to the definition $f$ is continuous and bounded and has continuous and bounded partial derivatives $\partial^jf$ up to the order $k$ and the functions $f^{(j)}:=\partial^jf$ for $|j|=k$ belong to the space $\Li(\alpha,\R^m)$. Moreover, when $k=0$, the Lipschitz class becomes the usual class  $\bC_b^{0,\alpha}(\E)$ of bounded H\"older continuous functions in $\E$,  which is consistent with \eqref{Def_Holder}.\\

A generalization of \eqref{Def_op_E0} is given by the linear extension operator 
\begin{equation}\label{Def_op_Ek}
{\cal{E}}_k(f^{(j)})(x)=\begin{cases}
			f^{(0)}(x),&x\in\E\\ {\sum\limits_{i}}'P(x,p_i)\varphi_i^*(x),&x\in\E^c
		\end{cases}
	\end{equation}
where $P(x,y)$ denotes the Taylor expansion of $f$ about $y\in\E,$ i.e. $P(x,y)=\displaystyle\sum_{|l|\le k}\frac{f^{(l)}(y)}{l !}(x-y)^{l}$. 

For $0<\alpha\le 1$, the operator ${\cal E}_k$ maps $\Li(k+\alpha,\E)$ continuously into $\Li(k+\alpha,\R^m)$ and the norm is independent of the closed set $\E$. \\

By using the Clifford norm $\|\cdot\|$ the $\R_{0,m}$-valued class of functions $\Li(k+\alpha,\E)$ may also be defined by means of the corresponding compatibility conditions
\begin{equation*}
	R_j(\ux,\uy)=f^{(j)}(\ux)-\sum_{|l|\le k-|j|}\frac{f^{(j+l)}(\uy)}{l !}(\ux-\uy)^{l},\,\,\ux, \uy\in\E
\end{equation*}
\begin{equation*}
	\|f^{(j)}(\ux)\|\le M,\;\|R_j(\ux,\uy)\|\le M|\ux-\uy|^{k+\alpha-|j|},\,\,\ux, \uy\in\E, |j|\le k,
\end{equation*}
where the functions $f^{(j)}$ and $R_j(\ux,\uy)$ are $\R_{0,m}$-valued as well.\\

As we mentioned before, the higher order Lipschitz class is connected with the works of H. Whitney, namely his celebrated extension theorem \cite[Thm. I]{Wh}, which turns out to be an appropriate way to define our multidimensional singular integral operator later. 
\begin{theorem}[Whitney Extension Theorem]
Let $f\in\Li(k+\alpha,\E)$.  Then, there exists a function $\f\in\Li(k+\alpha,\R^m)$ satisfying
\begin{itemize}
\item[(i)] $\f|_\E = f^{(0)}$,
\item[(ii)] $\partial^{j}\f|_\E=f^{(j)}$, $0<|j|\le k$,
\item[(iii)] $\f\in \bC^\infty(\R^{m} \setminus \E)$.
\end{itemize}
\end{theorem}

\section{Auxiliary results}
Our focus in this paper will be on the particular case of bimonogenic functions (i.e. $k$-monogenic with $k=2$), which are nothing more than $\R_{0,m}$-valued harmonic functions.\\

Bimonogenic functions enjoy a representation formula \cite[Theorem 7]{Ry}, namely if $f$ is bimonogenic in $\Omega$, then
\begin{equation}\label{repr_dom_+}
	f(\ux)=\int_\Gamma E_0(\uy-\ux)n(\uy) f(\uy)d\uy-\int_\Gamma E_1(\uy-\ux)n(\uy)\cD_{\uy} f(\uy)d\uy, x\in\Omega.
\end{equation}

Now assume $f\in \bC^2(\Omega_-)\cap \bC^{1}(\Omega_-\cup\Gamma)$ is bimonogenic in $\Omega_-$, $f(\infty)$ exists and $\cD_{\uy}f(\uy)=\so\big(\displaystyle\frac{1}{|\uy|}\big)$ as $|\uy|\to\infty$. Consider the ball $B_R(\ux)$ with center in $\ux\in\Omega_-$ and radius $R$ sufficiently big such that $\Omega_+\cup\Gamma\subset B_R(\ux)$, then if \eqref{repr_dom_+} is applied to the domain $B_R(\ux)\setminus \Omega_+\cup\Gamma$ we obtain
\[
f(\ux)=\int_{\Gamma^*} E_0(\uy-\ux)n(\uy) f(\uy)d\uy-\int_{\Gamma^*} E_1(\uy-\ux)n(\uy)\cD_{\uy} f(\uy)d\uy,
\]
where $\Gamma^*=-\Gamma\cup C_R(\ux)$ and $C_R(\ux)=\partial B_R(\ux)$. Then it follows from the identity
\[
\int_{C_R(\ux)} E_{0}(\uy-\ux)n(\uy) f(\uy)d\uy=\int_{C_R(\ux)} E_{0}(\uy-\ux)n(\uy)[f(\uy)-f(\infty)]d\uy+\Big(\int_{C_R(\ux)} E_{0}(\uy-\ux)n(\uy)d\uy\Big)f(\infty),
\]
the continuity of $f$ and the fact that $\int_{C_R(\ux)} E_{0}(\uy-\ux)n(\uy)d\uy=1$ that $$\int_{C_R(\ux)} E_{0}(\uy-\ux)n(\uy) f(\uy)d\uy\to f(\infty),\,\,\mbox{as}\,\,R\to \infty.$$ On the other hand, our assumption on the behavior at infinity implies $|\uy|\Vert\cD_{\uy} f(\uy)\Vert\leq \epsilon$ as $|\uy|\to \infty$ and it does so $\Vert\cD_{\uy} f(\uy)\Vert\to 0$ as $|\uy|\to\infty$, hence $|\uy-\ux|\Vert\cD_{\uy}f(\uy)\Vert\leq |\uy|\Vert\cD_{\uy}f(\uy)\Vert+|\ux|\Vert\cD_{\uy}f(\uy)\Vert\to 0$ as $|\uy|\to\infty$, that is, $\cD_{\uy} f(\uy)=\so\big(\frac{1}{|\uy-\ux|}\big)$ as $|\uy|\to \infty$ and therefore
\[
\Big\Vert\int_{C_R(\ux)} E_1(\uy-\ux)n(\uy) \cD_{\uy} f(\uy)d\uy\Big\Vert\leq \int_{C_R(\ux)} |E_1(\uy-\ux)|\Vert \cD_{\uy} f(\uy)\Vert d\uy\leq \frac{c\epsilon}{R^{m-1}}\int_{C_R(\ux)}d\uy\to 0 \,(\epsilon\to 0,\,R\to\infty).
\]

Under the above-mentioned assumptions, we arrive at a representation formula in the exterior domain
\begin{eqnarray}\label{repr_dom_-}
	f(\ux)=-\int_{\Gamma} E_0(\uy-\ux)n(\uy) f(\uy)d\uy+\int_{\Gamma} E_1(\uy-\ux)n(\uy)\cD_{\uy}f(\uy)d\uy+f(\infty),\,\,\ux\in\Omega_-.
\end{eqnarray}

Combining \eqref{repr_dom_+} with the Whitney extension theorem (component-wise applied to $\R_{0,m}$-valued functions) we introduce a Cauchy type transform of a first order Lipschitz function by
\begin{equation*}
[\cC_1 f]^{(0)}(\ux)=\int_\Gamma E_0(\uy-\ux)n(\uy)\f(\uy)d\uy-\int_\Gamma  E_1(\uy-\ux)n(\uy)\cD_{\uy}\f(\uy)d\uy,\,\,\ux\in\R^m\setminus\Gamma
\end{equation*}
and we do so its singular version
\begin{equation}\label{Def_S10}
[\cS_1f]^{(0)}(\uz)= 2\int_\Gamma E_0(\uy-\uz)n(\uy)\f(\uy)d\uy-\int_\Gamma  E_1(\uy-\uz)n(\uy)\cD_{\uy}\f(\uy)d\uy,\,\,\uz\in\Gamma
\end{equation}
where $\f$ denotes the $\R_{0,m}$-valued Whitney extension of $f\in\Li(1+\alpha,\Gamma)$. Note that $[\cC_1 f]^{(0)}$ is bimonogenic in $\R^m\setminus\Gamma$.\\

Due to the weakly singularity of the kernel $E_1$ we have from \eqref{cliff_Plemelj-Sokhotski_C0}
\begin{eqnarray}\label{cliff_Plemelj-Sokhotski_C1}
\begin{cases}
\big([\cC_1f]^{(0)}\big)^+(\uz)-\big([\cC_1f]^{(0)}\big)^-(\uz)=\f\vert_\Gamma=f^{(0)}(\uz),&\,\uz\in\Gamma\\
\big([\cC_1f]^{(0)}\big)^+(\uz)+\big([\cC_1f]^{(0)}\big)^-(\uz)=[\cS_1f]^{(0)}(\uz),&\,\uz\in\Gamma
\end{cases}
\end{eqnarray}
where, as usual, $\big([\cC_1f]^{(0)}\big)^\pm(\uz)=\lim\limits_{\ux\to \uz\atop \ux\in\Omega_\pm}[\cC_1f]^{(0)}(\ux).$\\

If we set 
\begin{eqnarray*}
[\cS_1f]^{(j)}(\uz)=2\int_\Gamma   E_0^{(j)}(\uy-\uz)n(\uy)R(\uy,\uz)d\uy-2\int_\Gamma   E_1^{(j)}(\uy-\uz)n(\uy)\cD_{\uy}R(\uy,\uz)d\uy+f^{(j)}(\uz),\,  |j|=1,
\end{eqnarray*}
then it follows that the first order Lipschitz class $\Li(1+\alpha,\Gamma)$ behaves invariant under the action of this singular integral operator \cite{RL}, that is, the Whitney data $\cS_1 f:=\{[\cS_1f]^{(j)},\,0\leq|j|\leq 1\}$ belongs to $\Li(1+\alpha,\Gamma)$, whenever $f\in\Li(1+\alpha,\Gamma)$. 
\\

We include for later reference the following lemma which was recently proved in \cite{AA}.
\begin{lemma}\label{Lfj}
Let $f,g\in\Li(1+\alpha,\Gamma)$ such that $f^{(0)}(\ux)=g^{(0)}(\ux)$ and $\sum\limits_{|j|=1}e_{(j)}f^{(j)}(\ux)=\sum\limits_{|j|=1}e_{(j)}g^{(j)}(\ux)$ for all $\ux\in\Gamma$. Then $f\equiv g$ in $\Gamma$, i.e. $f^{(j)}=g^{(j)}\,\,\forall |j|\leq k.$ 
\end{lemma}   

\section{Main results}\label{Sec_Main}
\begin{theorem}\label{Th_inv}
The singular integral operator $\cS_1$ is an involution. Namely,  
\begin{equation}\label{id_inv}
	[\cS_1^2f]^{(j)}=f^{(j)} \,\,\forall\,\,|j|\leq 1.
\end{equation}
\end{theorem}
\begin{proof} Let us first show that
\begin{eqnarray}\label{idSj}
\sum\limits_{|j|=1}e_{(j)}[\cS_1f]^{(j)}(\uz)=2\int_\Gamma  E_0(\uy-\uz)n(\uy)\Big(\sum\limits_{|j|=1}e_{(j)}f^{(j)}(\uy)\Big)d\uy.
\end{eqnarray}
Indeed,
\begin{eqnarray}\label{desglose}
\sum\limits_{|j|=1}e_{(j)}[\cS_1f]^{(j)}(\uz)&=&2\int_\Gamma  \sum\limits_{|j|=1}e_{(j)}E_0^{(j)}(\uy-\uz)n(\uy)R(\uy,\uz)d\uy-2\int_\Gamma  \sum\limits_{|j|=1}e_{(j)}E_1^{(j)}(\uy-\uz)n(\uy)\cD_{\uy}R(\uy,\uz)d\uy\nonumber\\
&&+\sum\limits_{|j|=1}e_{(j)}f^{(j)}(\uz)\nonumber\\
&=&2\int_\Gamma  \cD_{\uz}E_0(\uy-\uz)n(\uy)R(\uy,\uz)d\uy-2\int_\Gamma  \cD_{\uz} E_1(\uy-\uz)n(\uy)\cD_{\uy}R(\uy,\uz)d\uy\nonumber\\
&&+\sum\limits_{|j|=1}e_{(j)}f^{(j)}(\uz)\nonumber\\
&=&2\int_\Gamma  E_0(\uy-\uz)n(\uy)\cD_{\uy}R(\uy,\uz)d\uy+\sum\limits_{|j|=1}e_{(j)}f^{(j)}(\uz).
\end{eqnarray}	
Note that if $\exists i: j_i>l_i$, then $\partial_{\uy}^{(j)}(\uy-\uz)^{l}=0$. Thus, since $f^{(0)}(\uy)=f^{(0)}(\uz)+\displaystyle\sum\limits_{|l|=1}f^{(l)}(\uz)(\uy-\uz)^l+R(\uy,\uz)$, we have
\begin{eqnarray*}
\cD_{\uy}R(\uy,\uz)=\sum\limits_{|j|=1}e_{(j)}\partial_{\uy}^{(j)}\Big[\f(\uy)-\f(\uz)-\sum\limits_{|l|=1}\partial_{\uz}^{(l)}\f(\uz)(\uy-\uz)^l\Big]
=\sum\limits_{|j|=1}e_{(j)}\partial_{\uy}^{(j)}\f(\uy)-\sum\limits_{|j|=1}e_{(j)}\partial_{\uz}^{(j)}\f(\uz).
\end{eqnarray*}
When this is substituted in \eqref{desglose} we get
\begin{eqnarray*}
\sum\limits_{|j|=1}e_{(j)}[\cS_1f]^{(j)}(\uz)&=&2\int_\Gamma  E_0(\uy-\uz)n(\uy)\Bigg\{\sum\limits_{|j|=1}e_{(j)}\partial_{\uy}^{(j)}\f(\uy)-\sum\limits_{|j|=1}e_{(j)}\partial_{\uz}^{(j)}\f(\uz)\Bigg\}+\sum\limits_{|j|=1}e_{(j)}f^{(j)}(\uz).
\end{eqnarray*}	
Then \eqref{idSj} now follows from the above, since $\displaystyle\int_\Gamma  E_0(\uzz-\uy)n(\uzz)d\uzz=\frac{1}{2}$ when $\uy,\uzz\in\Gamma$.\\

Let us prove \eqref{id_inv} for $j=0.$
\begin{eqnarray*}
[\cS_1^2f]^{(0)}(\uz)&=&2\sum\limits_{s=0}^{1}\int_\Gamma   (-1)^sE_s(\uy-\uz)n(\uy)\cD_{\uy}^s\widetilde{\cS_kf}(\uy)d\uy\\
&=&2\int_\Gamma  E_0(\uy-\uz)n(\uy)\widetilde{\cS_1f}(\uy)d\uy-2\int_\Gamma  E_1(\uy-\uz)n(\uy)\cD_{\uy}\widetilde{\cS_1f}(\uy) d\uy\\
&=&2\int_\Gamma  E_0(\uy-\uz)n(\uy)\widetilde{\cS_1f}(\uy)d\uy-2\int_\Gamma  E_1(\uy-\uz)n(\uy)\sum\limits_{|j|=1}e_{(j)}\partial^{(j)}_{\uy}\widetilde{\cS_1f}(\uy)d\uy.
\end{eqnarray*}	

On substituting \eqref{Def_S10} and \eqref{idSj} together with the observation that $\puy^{(j)}\widetilde{\cS_1f}\vert_{\Gamma}=[\cS_1f]^{(j)}$ we obtain
\begin{eqnarray*}
[\cS_k^2f]^{(0)}(\uz)&=&2\int_\Gamma  E_0(\uy-\uz)n(\uy)[\cS_kf]^{(0)}(\uy)d\uy-2\int_\Gamma  E_1(\uy-\uz)n(\uy)\sum\limits_{|j|=1}e_{(j)}[\cS_1f]^{(j)}(\uy)d\uy\\
&=&2\int_\Gamma  E_0(\uy-\uz)n(\uy)\Big[2\int_\Gamma  E_0(\uzz-\uy)n(\uzz)\f(\uzz)d\uzz-2\int_\Gamma  E_1(\uzz-\uy)n(\uzz)\cD_{\uzz}\f(\uzz)d\uzz\Big]d\uy\\
&&-2\int_\Gamma  E_1(\uy-\uz)n(\uy)\Big[2\int_\Gamma  E_0(\uzz-\uy)n(\uzz)\Big(\sum\limits_{|j|=1}e_{(j)}f^{(j)}(\uzz)\Big)d\uzz\Big]d\uy.
\end{eqnarray*}	

Since $2\int_\Gamma  E_0(\uy-\uz)n(\uy)\Big[2\int_\Gamma  E_0(\uzz-\uy)n(\uzz)\f(\uzz)d\uzz\Big]d\uy=f^{(0)}(\uz)$ according to the classical involution property and $\sum_{|j|=1}e_{(j)}f^{(j)}(\uy)=\sum_{|j|=1}e_{(j)}\puy^{(j)} \f(\uy)=\cD_{\uy} \f(\uy),$ it is therefore enough to show that the following expression vanishes
\begin{eqnarray*}
Q&:=&2\int_\Gamma  E_0(\uy-\uz)n(\uy)\Big[2\int_\Gamma  -E_1(\uzz-\uy)n(\uzz)\cD_{\uzz}\f(\uzz)d\uzz\Big]d\uy-2\int_\Gamma  E_1(\uy-\uz)n(\uy)\Big[2\int_\Gamma  E_0(\uzz-\uy)n(\uzz)\cD_{\uzz}\f(\uzz)d\uzz\Big]d\uy\nonumber\\
&=&2\sum\limits_{s=0}^1 \int_\Gamma  (-1)^s E_s(\uy-\uz)n(\uy) \Big[2\int_\Gamma  -\cD_{\uy}^sE_1(\uzz-\uy)n(\uzz)\cD_{\uzz} \f(\uzz)d\uzz\Big] d\uy.
\end{eqnarray*}

Fubini's theorem leads to
\begin{eqnarray*}
Q&=&4\int_\Gamma  \int_\Gamma  \sum\limits_{s=0}^{1}(-1)^sE_s(\uy-\uz)n(\uy)\Big[-\cD_{\uy}^sE_1(\uzz-\uy)d\uy\Big]n(\uzz)\cD_{\uzz}d\uzz.
\end{eqnarray*}

Define $F(\uy):=-E_1(\uzz-\uy)$. It suffices to show that
\begin{eqnarray*}
\sum\limits_{s=0}^{1}\int_\Gamma  (-1)^sE_s(\uy-\uz)n(\uy)\cD_{\uy}^s F(\uy)d\uy=0.
\end{eqnarray*}	
Let $\Gamma^*=\Gamma\setminus(\Gamma_{\epsilon}(\uz)\cup\Gamma_{\epsilon}(\uzz))\cup -C^*_{\epsilon}(\uz)\cup -C^*_{\epsilon}(\uzz)$ and $\Omega_*$ its interior. Clearly, $F(\uy)$ is bimonogenic on $\Omega_*$ and continuous on $\Omega_*\cup\Gamma^*$, so we can apply Cauchy's formula (for $\uz\in\R^m\setminus\overline{\Omega}_*$) to get
\begin{eqnarray*}
\sum\limits_{s=0}^{1}\int_{\Gamma^*} (-1)^sE_s(\uy-\uz)n(\uy)\cD_{\uy}^s F(\uy)d\uy=0.
\end{eqnarray*}	
Then we are reduced to proving that
\begin{eqnarray}\label{lim}
\lim\limits_{\epsilon\to 0}
\Bigg\{\sum\limits_{s=0}^{1}\int_{C^*_{\epsilon}(\uz)\cup C^*_{\epsilon}(\uzz)} (-1)^sE_s(\uy-\uz)n(\uy)\cD_{\uy}^s F(\uy)d\uy\Bigg\}=0.
\end{eqnarray}
 
We compute
\begin{eqnarray}\label{e0e1}
&&\int_{C^*_{\epsilon}(\uz)\cup C^*_{\epsilon}(\uzz)}\bigg[E_0(\uy-\uz)n(\uy)E_1(\uzz-\uy)+E_1(\uy-\uz)n(\uy)E_0(\uzz-\uy)\bigg]d\uy =\nonumber\\
&=&c_m\Biggl\{    
\int_{C^*_{\epsilon}(\uz)}\frac{-(\uy-\uz)}{|\uy-\uz|^{m}}\frac{\uy-\uz}{|\uy-\uz|}\frac{1}{|\uzz-\uy|^{m-2}}d\uy+\int_{C^*_{\epsilon}(\uz)}\frac{1}{|\uy-\uz|^{m-2}}\frac{\uy-\uz}{|\uy-\uz|}\frac{-(\uzz-\uy)}{|\uzz-\uy|^{m}}d\uy\nonumber\\
&&+\int_{C^*_{\epsilon}(\uzz)}\frac{-(\uy-\uz)}{|\uy-\uz|^{m}}\frac{\uy-\uzz}{|\uy-\uzz|}\frac{1}{|\uzz-\uy|^{m-2}}d\uy+\int_{C^*_{\epsilon}(\uzz)}\frac{1}{|\uy-\uz|^{m-2}}\frac{\uy-\uzz}{|\uy-\uzz|}\frac{-(\uzz-\uy)}{|\uzz-\uy|^{m}}d\uy\Biggr\}
\nonumber\\
&=&\frac{c_m}{\epsilon^{m-1}}\Bigg\{\int_{C^*_{\epsilon}(\uz)}\frac{d\uy}{|\uzz-\uy|^{m-2}}-\int_{C^*_{\epsilon}(\uzz)}\frac{d\uy}{|\uy-\uz|^{m-2}}\nonumber\\
&&+\int_{C^*_{\epsilon}(\uz)}\frac{(\uy-\uz)(\uy-\uzz)}{|\uzz-\uy|^{m}}d\uy-\int_{C^*_{\epsilon}(\uzz)}\frac{(\uy-\uz)(\uy-\uzz)}{|\uy-\uz|^{m}}d\uy\Bigg\}.
\end{eqnarray}	 	

After the change of variable $\uy= \uz-\underline{t}+\uzz$, we get that the first and second expressions in curly brackets are equal. This observation applies to the third and fourth terms as well
and $[\cS_1^2f]^{(0)}=f^{(0)}$ as claimed. 
Now that we have the above claim, we set $G=\{[\cS_1^2f]^{(j)}(\uz)-f^{(j)}(\uz),\,|j|\leq 1\}$. Clearly, $G^{(0)}=0$. By Plemelj-Privalov's theorem \cite{RL}, $G\in\Li(1+\alpha,\Gamma)$ and according to Lemma \ref{Lfj} it suffices to show $\sum\limits_{|j|=1}e_{(j)}G^{(j)}=0.$

On applying \eqref{idSj} twice, we obtain
\begin{eqnarray*}
\sum\limits_{|j|=1}e_{(j)}G^{(j)}(\uz)	&=&
\sum\limits_{|j|=1}e_{(j)}\big[\cS_1(\cS_1f)\big]^{(j)}(\uz)-\sum\limits_{|j|=1}e_{(j)}f^{(j)}(\uz)\nonumber\\
&=&2\int_\Gamma  E_0(\uy-\uz)n(\uy)\Bigg(\sum\limits_{|j|=1}e_{(j)}(\cS_1f)^{(j)}(\uy)\Bigg)d\uy-\sum\limits_{|j|=1}e_{(j)}f^{(j)}(\uz)\nonumber\\
&=&2\int_\Gamma  E_0(\uy-\uz)n(\uy)\Bigg\{
2\int_\Gamma  E_0(\uzz-\uy)n(\uzz)\Big(\sum\limits_{|j|=1}e_{(j)}f^{(j)}(\uzz)\Big)d\uzz
\Bigg\}d\uy-\sum\limits_{|j|=1}e_{(j)}f^{(j)}(\uz).
\end{eqnarray*}

By the involution property of the classical singular operator, we get $\sum\limits_{|j|=1}e_{(j)}G^{(j)}(\uz)=0$, which completes the proof.
\end{proof} 
It is immediate that the operators $\cP^+=\frac{1}{2}(I+\cS_1)$ and $\cP^-=\frac{1}{2}(I-\cS_1)$
are projections on $\Li(1+\alpha,\Gamma)$, that is
\[
\cP^+\cP^+=\cP^+,\,\cP^-\cP^-=\cP^-,\,\cP^+\cP^-=0,\,\cP^-\cP^+=0.
\]
Consequently, $\Li(1+\alpha,\Gamma)=\Li^+(1+\alpha,\Gamma)\oplus\Li^-(1+\alpha,\Gamma),$
where $\Li^\pm(1+\alpha,\Gamma):=\mbox{im}\cP^\pm$. 

\begin{theorem} \label{Thm_Lip_plus}
The Whitney data $f\in\emph{\Li}(1+\alpha,\Gamma)$ belongs to $\emph{\Li}^+(1+\alpha,\Gamma)$ if and only if there exists a bimonogenic function $F$ in $\Omega_+$ which together with $\cD_{\ux}F$ continuously extends to $\Gamma$ and such that
\begin{equation}\label{trace+}
F|_\Gamma =f^{(0)},\,\,\cD_{\ux}F|_\Gamma =\sum\limits_{|j|=1}e_{(j)}f^{(j)}.
\end{equation}
\end{theorem}
\begin{proof}
The proof proceeds along the same lines as the proof of \cite[Theorem 6]{AA}, but we use \eqref{idSj} instead. Let us prove necessity. By definition, if $f\in\Li^+(1+\alpha,\Gamma)$ there exists $g\in\Li(1+\alpha,\Gamma)$ such that $f=\frac{1}{2}(g+\cS_1g)$, i.e.  
$f^{(j)}(\ux)=\frac{1}{2}[I^{(j)}+\cS_1^{(j)}]g, |j|\leq 1$, where $I^{(j)}g=g^{(j)}$ is the identity operator. Let us introduce the  function $F$ given by the Cauchy type transform $F(\ux)=[\cC_1g]^{(0)}(\ux),\,\ux\in\Omega_+$, which is clearly bimonogenic in $\Omega_+$. Thus, for $\uz\in\Gamma$ we get from \eqref{cliff_Plemelj-Sokhotski_C1} that
\[
F(\uz)=\lim_{\ux\to\uz\atop \ux\in\Omega_+}F(\ux)=
\big([\cC_1g]^{(0)}\big)^+(\uz)=\frac{1}{2}[I^{(0)}+\cS_1^{(0)}]g=f^{(0)}(\uz).
\]

On the other hand, 
\[
\cD_{\ux}F(\ux)=\cD_{\ux}\Big[\int_\Gamma  E_0(\uy-\ux)n(\uy)\g(\uy)d\uy-\int_\Gamma  E_1(\uy-\ux)n(\uy)\cD_{\uy}\g(\uy)d\uy\Big]=\cC_0[\cD\g](\ux)=\cC_0\Big[\sum\limits_{|j|=1}e_{(j)}\partial_{\ux}^{(j)}\g\Big](\ux).
\]
From \eqref{cliff_Plemelj-Sokhotski_C0} we have for $\uz\in\Gamma$
\begin{eqnarray*}
[\cD_{\ux}F](\uz)=\lim_{\ux\to\uz\atop \ux\in\Omega_+}\cD_{\ux}F(\ux)=\Big[\cC_0(\cD_{\ux}\g)\Big]^+(\uz)&=&\frac{1}{2}[I+\cS_0]\Big(\sum\limits_{|j|=1}e_{(j)}g^{(j)}\Big)(\uz)\\
&=&\frac{1}{2}\Big[\sum\limits_{|j|=1}e_{(j)}g^{(j)}(\uz)+2\int_\Gamma  E_0(\uy-\uz)n(\uy)\Big(\sum\limits_{|j|=1}e_{(j)}g^{(j)}(\uy)\Big)d\uy\Big].
\end{eqnarray*}
Hence it follows from \eqref{idSj} that $\cD_{\ux}F\vert_\Gamma=\sum\limits_{|j|=1}e_{(j)}f^{(j)}(\uz).$\\

We now prove sufficiency. Assume there exists such bimonogenic function $F$ satisfying \eqref{trace+}. It turns out that nothing more need be done to prove that $[\cP^+f]^{(j)}=f^{(j)}$ for all $0\le |j|\le 1$. Indeed, in that case, $f\in\mbox{im}\cP^+$ as claimed. 
Combining our assumptions with \eqref{repr_dom_+} gives
\begin{eqnarray*}
F(\ux)&=&\int_\Gamma  E_0(\uy-\ux)n(\uy)F(\uy)d\uy-\int_\Gamma  E_1(\uy-\ux)n(\uy)\cD_{\uy}F(\uy)d\uy,\,\,\ux\in\Omega_+\nonumber\\
&=&\int_\Gamma  E_0(\uy-\ux)n(\uy)f^{(0)}(\uy)d\uy-\int_\Gamma  E_1(\uy-\ux)n(\uy)\Big(\sum\limits_{|j|=1}e_{(j)}f^{(j)}(\uy)\Big)d\uy=[\cC_1f]^{(0)}(\ux).
\end{eqnarray*}
It is easy to see from this and \eqref{cliff_Plemelj-Sokhotski_C1} that
\begin{eqnarray*}
[\cP^+f]^{(0)}(\uz):=\frac{1}{2}[I^{(0)}+\cS_1^{(0)}]f(\uz)=\big([\cC_1f]^{(0)}\big)^+(\uz)=F(\uz)=f^{(0)}(\uz).
\end{eqnarray*}
On the other hand,
\begin{eqnarray}\label{ejP}
\sum\limits_{|j|=1}e_{(j)}[\cP^+f]^{(j)}(\uz)&:=&\sum\limits_{|j|=1}e_{(j)}\frac{1}{2}[I^{(j)}+\cS_1^{(j)}]f(\uz)\nonumber\\
&=&\frac{1}{2}\Big[\sum\limits_{|j|=1}e_{(j)}f^{(j)}(\uz)+\sum\limits_{|j|=1}e_{(j)}[\cS_1f]^{(j)}(\uz)\Big]\nonumber\\
&=&
\frac{1}{2}\Big[\sum\limits_{|j|=1}e_{(j)}f^{(j)}(\uz)+2\int_\Gamma  E_0(\uy-\uz)n(\uy)\Big(\sum\limits_{|j|=1}e_{(j)}f^{(j)}(\uy)\Big)d\uy\Big]
\end{eqnarray}
the last equality being a consequence of \eqref{idSj}. Since the density function in \eqref{ejP} satisfies the H\"older condition on $\Gamma$ and due to \eqref{trace+} it represents the interior limiting value of the monogenic function $\cD_{\ux} F$ then \eqref{cliff_Plemelj-Sokhotski_C0} 
yields
\[
\sum\limits_{|j|=1}e_{(j)}f^{(j)}(\uz)=\Big[\cC_0\big(\sum\limits_{|j|=1}e_{(j)}f^{(j)}\big)\Big]^+(\uz)
=\frac{1}{2}\big[I+\cS_0\big]\big(\sum\limits_{|j|=1}e_{(j)}f^{(j)}\big)(\uz),\]
or equivalently
\[
\frac{1}{2}\Big[\sum\limits_{|j|=1}e_{(j)}f^{(j)}(\uz)\Big]=\cS_0\Big(\sum\limits_{|j|=1}e_{(j)}f^{(j)}\Big)(\uz),
\]
which gives $$\sum\limits_{|j|=1}e_{(j)}[\cP^+f]^{(j)}(\uz)=\sum\limits_{|j|=1}e_{(j)}f^{(j)}(\uz)$$
when substituted in \eqref{ejP} and hence the assertion follows from Lemma \ref{Lfj}.
\end{proof}
Similar arguments to those above, but using \eqref{repr_dom_-} instead, show the following theorem and its proof is omitted.
\begin{theorem}\label{Thm_Lip_minus} 
The Whitney data $f\in\emph{\Li}(1+\alpha,\Gamma)$ belongs to $\emph{\Li}^-(1+\alpha,\Gamma)$ if and only if there exists a bimonogenic function $F$ in $\Omega_-$ vanishing at infinity satisfying $\cD_{\ux}F=\so\big(\displaystyle\frac{1}{|\ux|}\big)$ as $\ux\to\infty$, which together with $\cD_{\ux}F$ continuously extends to $\Gamma$ and such that
\begin{equation*}\label{trace-}
F|_\Gamma =f^{(0)},\,\,\cD_{\ux}F|_\Gamma =\sum\limits_{|j|=1}e_{(j)}f^{(j)}.
\end{equation*}
\end{theorem}

In other words, given $f\in\Li(1+\alpha,\Gamma)$, the problem of finding a sectionally bimonogenic $F$ satisfying the boundary conditions 
\[
\begin{cases}
F^+(\uz)-F^-(\uz)=f^{(0)}(\uz),&\,\uz\in\Gamma\\
[\cD_{\ux} F]^+(\uz)-[\cD_{\ux} F]^-(\uz)=\sum\limits_{|j|=1}e_{(j)}f^{(j)}(\uz),&\,\uz\in\Gamma\\
F(\infty)=0,\,\,\cD_{\ux}F=\so\big(\dfrac{1}{|\ux|}\big) ,& {\mbox{as}}\,\, \ux\to\infty
\end{cases}
\]
has a unique solution and it is expressed by $F=[\cC_1f]^{(0)}.$
\section{Concluding remarks}
In the case of considering $k$-monogenic functions with $k>1$, a general representation formula allowed us in \cite{RL} to define the associated Cauchy transform and the singular integral operator $\cS_k$ related to Lipschitz functions of arbitrary order $\Li(k+\alpha,\Gamma)$. We can ask whether Theorems \ref{Th_inv} - \ref{Thm_Lip_minus} are still valid for polymonogenic functions and Lipschitz classes of arbitrary order. Partial evidence support the belief that the method developed in this paper can also be successfully applied to the general case. The confirmation of this hypothesis inspires further analysis and research, which will be the subject of future work.
\section{Acknowledgments}
L. De la Cruz Toranzo was supported by a Research Fellowship under the auspices of the Alexander von Humboldt Foundation.

\end{document}